\documentclass[runningheads]{llncs}
\usepackage[T1]{fontenc}
\usepackage{graphicx}
\usepackage{amssymb}
\usepackage{amsmath}

\usepackage{hyperref}

\usepackage{algorithm}
\usepackage{algpseudocode}

\begin{document}
\title{A Fully Adaptive Frank-Wolfe Algorithm for Relatively Smooth Problems and Its Application to Centralized Distributed Optimization.}
%
%
\author{A.A. Vyguzov\inst{1,3,4}\orcidID{0009-0005-1681-1750} \and
	F.S. Stonyakin \inst{1,2,3}\orcidID{0000-0002-9250-4438}}
\authorrunning{A. Vyguzov et al.}
%
\institute{Moscow Institute of Physics and Technology, Dolgoprudny, Institutsky lane, 9, Russia \and
	Simferopol, Academician Vernadsky Avenue, 4, V. I. Vernadsky Crimean Federal University, Republic of Crimea, Russia \and Innopolis University, Kazan, Tatarstan, 420500 Russia \and Adyghe State University, Maikop, Russia, Pervomayskaya, 208}

\maketitle              
\begin{abstract}
	We study the Frank-Wolfe algorithm for constrained optimization problems with relatively smooth objectives. Building upon our previous work, we propose a fully adaptive variant of the Frank-Wolfe method that dynamically adjusts the step size. Our method does not require prior knowledge of the function parameters and guarantees convergence using only local information. We establish a linear convergence rate under relative strong convexity and provide a detailed theoretical analysis of the proposed adaptive step-size rule.
	
	Furthermore, we demonstrate how relative smoothness and strong convexity naturally arise in the setting of centralized distributed optimization. Under a variance-type assumption on the gradients, we show that the global objective becomes relatively strongly convex with respect to the Bregman divergence generated by a local function. This structure allows us to apply our adaptive Frank-Wolfe algorithm, leading to provable acceleration due to an improved relative condition number.
	
	
\end{abstract}

\section{Introduction}

We consider the following constrained optimization problem:

\begin{align}\label{main_problem}
	\min_{x \in Q} \quad  f(x),
\end{align}

where $Q$ is a convex and compact subset of $\mathbb{R}^n$, and $f: Q \to \mathbb{R}$ is a convex function. We denote by $x^*$ and $f^*$ the solution and optimal value of problem \eqref{main_problem}, respectively.

We assume that $f$ in \eqref{main_problem} is relatively smooth (\cite{bauschke2017descent}, \cite{lu2018relatively}) with smoothness constant $L > 0$, meaning it satisfies the following condition:
\begin{equation}\label{def:rs}
	f(x) \leq f(y) + \langle \nabla f(y), x - y\rangle + L V(x, y), \quad \forall x, y \in Q,
\end{equation}
where $V(x, y)$ is the Bregman divergence between $x$ and $y$, defined as \( V(x, y) = h(x) - h(y) - \langle \nabla h(y), x - y \rangle \). Here, $h$ is a convex (not necessarily strongly convex) reference function.

Furthermore, we assume that the divergence function \( V(x, y) \) in \eqref{def:rs} satisfies the Triangle Scaling Exponent (TSE) property with some constant \( \gamma \in (1,2] \), as introduced in \cite{hanzely2021accelerated}:  
\begin{equation}\label{def:tse}
	V((1 - \theta)x + \theta z, (1 - \theta)x + \theta \tilde{z}) \leq \theta^\gamma V(z, \tilde{z}), \quad \forall \theta \in [0,1], \quad \forall x, z, \tilde{z} \in Q.
\end{equation}

The TSE property plays a key role in the convergence analysis of optimization algorithms under relative smoothness. The value of \( \gamma \) depends on the choice of divergence. For instance, for the Euclidean distance (a special case of Bregman divergence) or the Kullback–Leibler divergence, we have \( \gamma = 1 \). For other divergences, \( \gamma \) may not be explicitly known. Notably, for the Itakura–Saito (IS) divergence, it is known that \( \gamma > 0.5 \), although the exact value remains unspecified.

In this paper, we continue the development of the Frank-Wolfe method for relatively smooth objectives, initiated in \cite{vyguzov2024adaptivevariantfrankwolfemethod}. We propose a modified adaptive step size for the Bregman setting, which depends not only on the parameter $L$ from \eqref{def:rs}, but also on the TSE exponent $\gamma$ from \eqref{def:tse}:
\begin{equation}\label{step_size_l_smooth}
	\alpha_k := \min \left\{ \left( \frac{- \langle \nabla f(x_k), d_k \rangle }{2 L_k V(s_k, x_k)} \right) ^{\frac{1}{\gamma_k - 1}}, 1 \right\}.
\end{equation}

The Frank-Wolfe update step is given by:
\begin{equation}\label{FW_step_size}
	x_{k+1} = x_k + \alpha_k d_k,
\end{equation}
where $\mathrm{LMO}_Q(g) = \arg\min_{z \in Q} g^\top z$ is the linear minimization oracle (LMO), and $d_k = s_k - x_k$ is the descent direction, with $s_k \in \mathrm{LMO}_Q(\nabla f(x_k))$.

A similar idea for adapting the step size was explored in \cite{takahashi2025fast} and \cite{pedregosa2020linearly}, but our condition for adaptation and the corresponding procedure differ, which leads to distinct convergence behavior and analysis.

We establish a linear convergence rate for the proposed method under relative strong convexity. This analysis has not been presented in prior works. 

\begin{definition}
	A function $f$ is said to be \emph{relatively strongly convex} if it is differentiable on $Q$ and there exists $\mu > 0$ such that the following condition holds:
	\begin{equation}\label{rel_strong_conv}
		f(x) \geq f(y) + \langle \nabla f(y), x - y\rangle + \mu L V(x, y), \quad \forall x, y \in Q.
	\end{equation}
\end{definition}

The Frank-Wolfe method is well-suited for distributed optimization problems due to its simpler implementation compared to gradient-based methods. Constrained problems over structured domains also naturally occur in many applications; see, for instance, the survey \cite{ghojogh2021kkt}.

Our result enables the application of the method to centralized distributed optimization problems. Specifically, we consider the following objective:
\begin{equation}\label{distr_opt_problem}
	F(x) = \frac{1}{m} \sum_{j=1}^{m} f_j(x) \to \min,
\end{equation}
where each function \( f_j \) corresponds to a node (e.g., a computer) in a network, which computes \( \nabla f_j \). We assume the existence of a central node that maintains access to the aggregate gradient:
\begin{equation}\label{distr_opt_central_node}
	\tilde{F}(x) = \frac{1}{\tilde{m}} \sum_{l=1}^{\tilde{m}} f_l(x) \quad \text{(or simply \( f_1(x) \))}.
\end{equation}

The remaining nodes transmit their gradient information \( \nabla f_j \) to the central node. The central node thus has access to \( \nabla F(x) \), but not necessarily to the full function value \( F(x) \); it may only know \( \tilde{F}(x) \).

Additionally, in our analysis we employ the assumption of so-called \emph{statistical similarity}:

\begin{equation*}
	\|\nabla F(x) - \nabla \tilde{F}(x) - \nabla F(y) + \nabla \tilde{F}(y)\|^2 \leq \sigma \|x - y\|^2, \quad \forall x, y,
\end{equation*}
where $\sigma > 0$ quantifies the degree of similarity between the Hessian matrices of the local and global loss functions.

Statistical similarity is a widely used technique for accelerating distributed optimization methods (see \cite{shamir2014communication}, \cite{arjevani2015communication}, \cite{rogozin2021towards}, \cite{tian2022acceleration}; and in the context of variational inequalities, see \cite{beznosikov2023similarity}, \cite{beznosikov2022compression}). These works leverage the similarity structure of the data to obtain improved theoretical guarantees. However, they do not address constrained problems or the Frank-Wolfe method.

The idea of improving the condition number via relative smoothness and statistical similarity was initially introduced in \cite{hendrikx2020statistically}, where it was applied to empirical risk minimization using gradient-based algorithms. In this work, we extend this approach to the Frank-Wolfe method.

\subsection{Contributions and Outline}

Our main contributions are as follows:
\begin{itemize}
	\item We propose a fully adaptive Frank-Wolfe method with Bregman step size for relatively smooth problems. The method does not rely on prior knowledge of problem-specific parameters.
	\item A linear convergence rate is established for relatively smooth and relatively strongly convex problems, given additional assumptions on the exact solution.
	\item We theoretically demonstrate that the proposed method outperforms its Euclidean counterpart in the centralized distributed optimization setting.
	\item We provide empirical evidence that our method outperforms non-adaptive and partially adaptive variants.
\end{itemize}

\subsection{Related works}

The Frank-Wolfe algorithm originally introduced in \cite{frank1956algorithm}, also known as the Conditional Gradient Method \cite{levitin1966constrained}. Unlike classical projected gradient methods, Frank-Wolfe replaces the often costly projection step with a cheaper linear minimization oracle making it especially attractive in large-scale applications where projection is prohibitively expensive (see more on this in \cite{combettes2021complexity} and \cite{braun2022conditional} Table 1.1). Furthermore, the algorithm has been successfully applied to distributed optimization frameworks with decentralized architectures \cite{wai2017decentralized}.

The classical smoothness condition, characterized by the Lipschitz gradient property, is central to the development and analysis of first-order optimization methods. However, many applications involve objective functions that, despite being convex and differentiable, do not satisfy this smoothness assumption. In \cite{bauschke2017descent}, the concept of relative smoothness was introduced, significantly expanding the analytical framework for a broad class of functions. Many functions that are not smooth in the classical sense are relatively smooth, including the D-optimal design problem, the Poisson linear inverse problem, and even simple univariate functions such as \( f(x) = x^4 \) or \( f(x) = -\log{x} + x^2 \) (see \cite{bauschke2017descent}, \cite{lu2018relatively}, \cite{dragomir2021quartic} for further examples).

The notion of relative smoothness has also found applications in distributed optimization problems. For instance, in empirical risk minimization \cite{hendrikx2020statistically}, it has been shown that, under the assumption of statistical similarity, the objective becomes relatively smooth with a known constant, significantly reducing the condition number. In this setting, alongside accelerated theoretical guarantees, the work \cite{hendrikx2020statistically} also proposes an adaptive algorithm, SPAG. In our work we extend this idea on Frank-Wolfe method.

\subsection{Frank-Wolfe gap}
Before presenting the main convergence estimates, we first recall the definition of one of the most important properties of the Frank-Wolfe algorithm.

\begin{definition}
	The Frank-Wolfe gap at iteration \( k \) is defined as
	\begin{equation*}\label{fw_gap}
		G(x) = \max_{s \in Q} - \langle \nabla F(x_k), s - x_k \rangle.
	\end{equation*}
\end{definition}

This quantity is a key parameter in the Frank-Wolfe algorithm (see \cite{bomze2021frank}), serving as a measure of optimality and a potential stopping criterion:
\begin{equation}\label{fw_gap_ineq}
	G(x) \geq -\nabla F(x)^\top (x^* - x) \geq F(x) - F^*.
\end{equation}

\section{Frank-Wolfe Method for Relatively Smooth Objectives with Fully Adaptive Step Size}

In this section, we introduce a Frank-Wolfe method with full parameter adaptation for relatively smooth objectives. As in our previous work \cite{vyguzov2024adaptivevariantfrankwolfemethod}, we select the shortest step size. However, in contrast to that work, both the smoothness constant \( L \) from \eqref{def:rs} and the TSE exponent \( \gamma \) from \eqref{def:tse} are now adapted at each iteration:

\begin{equation*}
	\alpha_k := \min \left\{ \left( \frac{- \langle \nabla f(x_k), d_k \rangle }{2 L_k V(s_k, x_k)} \right) ^{\frac{1}{\gamma_k - 1}}, 1 \right\}.
\end{equation*}

We provide convergence guarantees for this method and show that the adaptive step-size strategy preserves the convergence rate of the algorithm.

\begin{algorithm}  
	\caption{Frank-Wolfe Method with Fully Adaptive Step Size}  
	\label{alg:exp_adapt_simplified}  
	\begin{algorithmic}[1]  
		\Require Initial point \( x_0 \in \operatorname{rint} Q \), \( \gamma_0 = 2 \), \( L_0 > 0 \), \( \delta > 0 \), \( \eta > 1 \)
		\State Initialize: \( x_0 \), \( \gamma_{-1} = \gamma_0 \), \( L_{-1} = L_0 \)  
		\For{\( k = 0, 1, 2, \dots \)}  
		\State \( L_k \gets L_{k-1} / 2 \), \( \gamma_k \gets \min \{ \gamma_k + \eta (\gamma_k - 1), \gamma_{\max} \} \)
		\While{true}
		\State \( s_k \in \text{LMO}_Q (\nabla f(x_k)) \)  
		\State \( d_k = s_k - x_k \)  
		\State \( \alpha_k := \min \left\{ \left( \frac{- \langle \nabla f(x_k), d_k \rangle }{2 L_k V(s_k, x_k)} \right) ^{\frac{1}{\gamma_k - 1}}, 1 \right\} \)  
		\If{\( f(x_k + \alpha_k d_k) \leq f(x_k) + \alpha_k \langle \nabla f(x_k), d_k \rangle + \alpha_k^{\gamma_k} L_k V(s_k, x_k) \)}  \label{alg:exp_adapt_simplified:cond}
		\State \textbf{break}  
		\ElsIf{\( k \ \% \ 2 == 0 \)}  
		\State \( L_k \gets 2 L_k \)
		\Else  
		\State \( \gamma_k \gets 1 + \frac{\gamma_k - 1}{\eta} \)
		\EndIf  
		\EndWhile
		\State \( x_{k+1} \gets x_k + \alpha_k d_k \)
		\EndFor  
	\end{algorithmic}  
\end{algorithm}

\subsection{Sublinear Convergence Rate}

In this section, we analyze the general case under minimal assumptions.

Our analysis employs a standard technique that minimizes an approximation of the objective function rather than the exact function. To that end, we present two auxiliary lemmas.

\begin{lemma}
	Suppose that \( f \) satisfies relative smoothness \eqref{def:rs} and has the TSE property \eqref{def:tse} with some global \( L \) and \( \gamma \in (1, 2] \); then:
	\begin{equation}\label{lemma:rs_tse}
		f(x_k + \alpha_k d_k) \leq f(x_k) + \alpha_k \nabla f(x_k)^\top d_k + \alpha_k^\gamma L V(x_k + d_k, x_k).
	\end{equation}
\end{lemma}

\begin{proof}
	Applying \eqref{def:tse}, we obtain:
	\begin{gather*}
		V(x_k + \alpha_k d_k, x_k) = V((1 - \alpha_k)x_k + \alpha_k(x_k + d_k), (1 - \alpha_k)x_k + \alpha_k x_k) \leq \\
		\leq \alpha_k^\gamma V(x_k + d_k, x_k) = \alpha_k^\gamma V(s_k, x_k).
	\end{gather*}
	This completes the proof.
\end{proof}

Lemma \ref{lemma:rs_tse} ensures that the condition in line \ref{alg:exp_adapt_simplified:cond} of the algorithm is satisfiable. Throughout the remainder of the text, we assume that \ref{lemma:rs_tse} holds for some local \( L_k \) and \( \gamma_k \).

\begin{lemma}[Progress lemma]\label{descent_lemma}
	Suppose \( f \) satisfies relative smoothness \eqref{def:rs} and the triangle scaling property \eqref{def:tse}. Then, for each iteration \( k \) of Algorithm~\ref{alg:exp_adapt_simplified}, we have:
	\begin{equation}\label{lemma:descent_alpha_less_1}
		f(x_{k+1}) - f(x_k) \leq \nabla f(x_k)^\top d_k \cdot \min \left\{ \frac{1}{2}, \ \frac{1}{2} \left( \frac{- \nabla f(x_k)^\top d_k}{ 2 L_{\max} V(s_k, x_k) } \right)^{1/(\gamma_{\min} - 1)} \right\},
	\end{equation}
	where $L_{\max} = \max\limits_{i=0,\dots,k} L_i$ and $\gamma_{\min} = \min\limits_{i=0,\dots,k} \gamma_i$.
\end{lemma}

Note that the left branch of the minimum in Lemma~\ref{descent_lemma} applies when the step size \eqref{step_size_l_smooth} equals 1, and the right branch otherwise.

It is worth to note that left side of the min operator in the lemma~\ref{descent_lemma} is satisfied when step size \eqref{step_size_l_smooth} equals 1 and the right side otherwise.

\begin{proof}
	If \( \alpha_k = 1 \), then by \eqref{lemma:rs_tse}:
	\begin{gather*}
		f(x_{k+1}) \leq f(x_k) + \nabla f(x_k)^\top d_k + L_k V(s_k, x_k),
	\end{gather*}
	and thus:
	\begin{gather*}
		f(x_{k+1}) - f(x_k) \leq \nabla f(x_k)^\top d_k + L_k V(s_k, x_k) \leq \frac{1}{2} \nabla f(x_k)^\top d_k,
	\end{gather*}
	where the last inequality uses the step-size definition \eqref{step_size_l_smooth}.
	
	If \( \alpha_k < 1 \), then substituting the step size from \eqref{step_size_l_smooth} into \eqref{lemma:rs_tse}, we get:
	\begin{equation*}
		\begin{aligned}
			& f(x_{k+1}) - f(x_k) \leq \\ 
			&\quad - \frac{(- \nabla f(x_k)^\top d_k)^{\gamma_k / (\gamma_k - 1)} }{(2 L_k V(s_k, x_k))^{1 / (\gamma_k - 1)}} + \frac{(- \nabla f(x_k)^\top d_k )^{\gamma_k / (\gamma_k - 1)}}{2^{\gamma_k / (\gamma_k - 1)} (L_k V(s_k, x_k))^{1 / (\gamma_k - 1)}} \\
			&\quad = - \frac{1}{2} \frac{(-\nabla f(x_k)^\top d_k )^{\gamma_k / (\gamma_k - 1)}}{(2L_k V(s_k, x_k))^{1 / (\gamma_k - 1)}} \\
			&\quad = \frac{1}{2} \nabla f(x_k)^\top d_k \left(  \frac{-\nabla f(x_k)^\top d_k }{2L_k V(s_k, x_k) } \right)^{1 / (\gamma_k - 1)}.
		\end{aligned}
	\end{equation*}
	This proves the lemma.
\end{proof}

To prove the convergence rate of the Algorithm \eqref{alg:exp_adapt_simplified}, we will need the next auxiliary lemma.

\begin{lemma}\label{lemma:induction_aux_lemma}
	Suppose we are given a sequence \( \{h_k\} \) satisfying the recurrence
	\begin{equation*}
		h_{k+1} \leq h_k \left(1 - \left(\frac{h_k}{K}\right)^{1/(\gamma - 1)} \right),
	\end{equation*}
	where \( K > h_k > 0 \) and \( \gamma \in (1, 2] \). Then it holds that
	\begin{equation*}
		h_k \leq \left(\frac{2}{k+2}\right)^{\gamma - 1} K.
	\end{equation*}
\end{lemma}

\begin{proof}
	We consider two cases. 
	
	\textbf{Case 1:} Suppose
	\[
	\left(\frac{h_k}{K}\right)^{1/(\gamma - 1)} < \frac{2}{k+2}.
	\]
	Then it follows directly that
	\[
	h_k < \left(\frac{2}{k+2}\right)^{\gamma - 1} K,
	\]
	which is precisely the desired bound.
	
	\textbf{Case 2:} Suppose instead that
	
	\[
	\left(\frac{h_k}{K}\right)^{1/(\gamma - 1)} \geq \frac{2}{k+2},
	\]
	which implies
	\[
	h_k \geq \left(\frac{2}{k+2}\right)^{\gamma - 1} K.
	\]
	Using the assumed recurrence, we obtain
	\begin{equation*}
		\begin{aligned}
			& h_{k+1} \leq h_k \left(1 - \left(\frac{h_k}{K}\right)^{1/(\gamma - 1)} \right)
			\leq h_k \left(1 - \frac{2}{k+2} \right) \leq \left(\frac{2}{k+2}\right)^{\gamma - 1} K \\
			& \left(1 - \frac{2}{k+2} \right) \leq \left(\frac{2}{k+2}\right)^{\gamma - 1} K .\\
		\end{aligned}
	\end{equation*}
	
	Where the last inequality since \( \gamma \in (1,2] \).
	
	This proves the lemma.
\end{proof}

Now we prove the convergence rate of Algorithm~\ref{alg:exp_adapt_simplified}.

\begin{theorem}
	Suppose that for all \( x, y \in Q \), the Bregman divergence \( V(x, y) \) satisfies \( V(x,y) \leq \frac{R^2}{2} \) and the triangle scaling property \eqref{def:tse} holds for some unknown \( \gamma \in (1,2] \). If \( f \) is \( L \)-relatively smooth, then Algorithm~\ref{alg:exp_adapt_simplified} achieves the following convergence rate:
	\begin{equation}\label{theorem:alg_exp_adapt}
		f(x_k) - f^* \leq \left( \frac{2}{k+2} \right)^{\gamma_{\min} - 1} L_{\max} R^2, \quad \forall k \geq 1,
	\end{equation}
	
	where $L_{\max} = \max\limits_{i=0,\dots,k} L_i$ and $\gamma_{\min} = \min\limits_{i=0,\dots,k} \gamma_i$.
\end{theorem}

\begin{proof}
	We consider two cases based on the step size \eqref{step_size_l_smooth}: when it equals 1 and when it is less than 1.
	
	\textbf{Case 1 ($\alpha < 1$)}: Using the right side of the min operator in Lemma~\ref{descent_lemma}, we obtain:
	
	\begin{equation*}
		f(x_{k+1}) - f(x_k) \leq \frac{\nabla f(x_k)^\top d_k}{2} \left( \frac{- \nabla f(x_k)^\top d_k}{ 2 L_k V(s_k, x_k) } \right)^{1/(\gamma_k - 1)}.
	\end{equation*}
	After simplification, this yields:
	\[
	\begin{aligned}
		f(x_{k+1}) - f^* &\leq (f(x_k) - f^*) + \frac{\nabla f(x_k)^\top d_k}{2} 
		\left( \frac{- \nabla f(x_k)^\top d_k}{ 2 L_k V(s_k, x_k) } \right)^{1/(\gamma_k - 1)} \\
		&\overset{\textcircled{1}}{\leq} (f(x_k) - f^*)\left( 1 - \left( \frac{f(x_k) - f^*}{2 L_k V(s_k, x_k)} \right)^{1/(\gamma_k - 1)} \right) \\
		&\leq (f(x_k) - f^*)\left( 1 - \left( \frac{f(x_k) - f^*}{2 L_\text{max} V(s_k, x_k)} \right)^{1/(\gamma_{\min} - 1)} \right),
	\end{aligned}
	\]
	
	where \textcircled{1} follows from the FW-gap property \eqref{fw_gap_ineq}.
	
	Applying Lemma~\ref{lemma:induction_aux_lemma} with $h_{k+1} = f(x_{k+1}) - f^*$, $h_k = f(x_k) - f^*$, and $K = 2 L_k V(s_k, x_k)$ yields the desired estimate.
	
	\textbf{Case 2 ($\alpha_k = 1$)}: We proceed by induction. 
	
	\textit{Base case ($k=1$)}: Using the left side of the min operator in Lemma~\ref{descent_lemma}:
	\begin{equation*}
		\begin{aligned}
			f(x_1) - f^* &\leq f(x_0) - f^* + \langle \nabla f(x_0), x_1 - x_0 \rangle + L_0 V(x_1, x_0) \\
			&\leq f(x_0) - f^* + \langle \nabla f(x_0), x^* - x_0 \rangle + L_0 V(x_1, x_0) \\
			&\overset{\textcircled{1}}{\leq} L_0 V(x_1, x_0) \leq \frac{L_0 R^2}{2} \leq \frac{2L_0 R^2}{3} \\
			&\leq \left( \frac{2}{1 + 2} \right)^{\gamma_0 - 1} L_0 R^2.
		\end{aligned}
	\end{equation*}
	
	Inequality \textcircled{1} follows from the convexity of $f$, specifically $f(x_0) - f^* + \langle \nabla f(x_0), x^* - x_0\rangle \leq 0$, while the final inequality holds since $0 < \gamma_0 - 1 \leq 1$.
	
	\textit{Inductive step}: Assume $f(x_k) - f^* \leq \left( \frac{2}{k+2} \right) ^{\gamma_k - 1} L_k R^2$ from either the base case or Case 1. We must show:
	\begin{equation*}
		\begin{aligned}
			f(x_{k+1}) - f^* \leq \left( \frac{2}{k+3} \right)^{\gamma_{\min} - 1} L_{\max} R^2.
		\end{aligned}
	\end{equation*}
	
	Applying $\alpha_k = 1$ and the FW-gap \eqref{fw_gap_ineq} to \eqref{descent_lemma} gives:
	\begin{equation*}
		\begin{aligned}
			f(x_{k+1}) - f^* &\leq \frac{f(x_k) - f^*}{2} < \left( \frac{k+2}{k+3} \right)^{\gamma_k - 1}(f(x_k) - f^*) \\
			&\leq \left(\frac{2}{k+2} \right)^{\gamma_k - 1} \left(\frac{k + 2}{k+3} \right)^{\gamma_k - 1} L_k R^2 \\
			&\leq \left(\frac{2}{k+3} \right)^{\gamma_{\min} - 1}  L_\text{max} R^2.
		\end{aligned}
	\end{equation*}
	
	This completes the inductive step, establishing the claimed inequality for all iterations $k$.
	
	The theorem is proven.
\end{proof}

\subsection{Analyzes of adaptation steps}

This approach provides two key advantages. First, in practice, it can lead to acceleration because as demonstrated in \cite{hanzely2021accelerated}, a larger TSE can results in a better convergence rate for gradient method. Second, this method eliminates the need for prior knowledge of the exact TSE value. 

We will now show that adaptivity does not affect the convergence rate of Algorithm \ref{alg:exp_adapt_simplified} in any way.


\begin{remark}\label{remark:rs_adapt}
	If $f$ has $L$ relative-smoothness constant and satisfy TSE property with $\gamma$, then through Algorithm~\ref{alg:exp_adapt_simplified} at iteration $k$ we have $L_k \leq L_{\max}$ and $\gamma - 1 \geq \gamma_{\min} - 1$, where $L_{\max} = \max\limits_{i=0,\dots,k} L_i$ and $\gamma_{\min} = \min\limits_{i=0,\dots,k} \gamma_i$.
\end{remark}

\begin{remark}\label{remark:checks_amount}
	Assume that $f$ is a relatively smooth function with constant $L$. Suppose further that the function possesses the TSE property with constant $\gamma$. Let $i_k$ denote the number of line~\ref{alg:exp_adapt_simplified:cond} checks in Algorithm~\ref{alg:exp_adapt_simplified} at iteration $k$. Then the total number of line~\ref{alg:exp_adapt_simplified:cond} checks will be:
	\begin{align*}
		\sum_{k=0}^N i_k &= \sum_{k=1}^N \left(3 + \log_2\frac{L_k}{L_{k-1}} + \log_\eta\frac{\gamma_{k-1}-1}{\gamma_k-1}\right) \\
		&= 3N + \log_2\frac{L_N}{L_0} + \log_\eta\frac{\gamma_0-1}{\gamma_N-1} \\
		&\leq 3N + \log_2\frac{2L_{\max}}{L_0} + \log_\eta\frac{\gamma_0-1}{\gamma_{\min}-1} = O(N),
	\end{align*}
	where $\eta > 1$ is the constant from Algorithm~\ref{alg:exp_adapt_simplified}.
	
	Thus, the total number of inequality checks from line~\ref{alg:exp_adapt_simplified:cond} after $N$ iterations is $O(N)$ and is the same as initial algorithm iterations number.
\end{remark}

\subsection{Linear convergence rate}

In this section, we establish the linear convergence rate of Algorithm~\ref{alg:exp_adapt_simplified}. To this end, we adapt standard definitions to the framework of relative smoothness and introduce several necessary concepts. 

Next, we demonstrate how relative smoothness significantly improves the convergence rate in distributed optimization.

We now introduce the scaling condition adapted to the Bregman divergence.

\begin{definition}
	The \emph{scaling condition} is defined as
	\begin{equation} \label{scaling_cond_def}
		\begin{aligned}
			\frac{ - \nabla f(x_k)^\top d_k}{V(s_k, x_k)} \geq \tau \ (-\nabla f(x_k))^\top \left( \frac{x^* - x_k}{V(x^*, x_k)} \right)
		\end{aligned}
	\end{equation}
	for a fixed constant $\tau > 0$ and $x^* \in \text{argmin}_{x \in Q} f(x)$.
\end{definition}

Note the similarity of this property to the standard Euclidean case (see, for example, \cite{braun2022conditional}, Proposition 2.16).

We now determine the parameter $\tau$ in Definition~\ref{scaling_cond_def}.

\begin{lemma}\label{scaling_condition_determined}
	Suppose $f$ is a convex function, and let $x^*$ be an interior point of $Q$ in the Euclidean sense. Then, for Algorithm~\ref{alg:exp_adapt_simplified}, the scaling condition takes the form:
	\begin{equation}
		\begin{aligned}
			\frac{ - \nabla f(x_k)^\top d_k}{V(s_k, x_k)} \geq \frac{\delta}{D} \cdot \frac{\epsilon}{D_V} \cdot (-\nabla f(x_k))^\top \left( \frac{x^* - x_k}{V(x^*, x_k)} \right),
		\end{aligned}
	\end{equation}
	where $\epsilon$ is the desired residual, $\partial Q$ is a boundary of the feasible set, $\delta = \text{dist}(x^*, \partial Q) = \inf_{y \in \partial Q} \| x^* - y \|_2$, $D = \max_{x,y} \| x - y\|_2$, and $D_V = \max_{x,y} V(x, y)$ for $x, y \in Q$.
\end{lemma}

\begin{proof}
	Let $g = -\nabla f(x_k)$ and $\hat{g} = \frac{g}{\| g \|_2}$. Since $x^* \in \text{Int}(Q)$, we have $x^* + \delta \hat{g} \in Q$. Therefore,
	\begin{equation}\label{g_norm_ineq}
		\begin{aligned}
			g^{\top} d_k &\geq g^{\top}\left(\left(x^*+\delta \widehat{g}\right)-x_k\right) \\
			&= \delta g^{\top} \widehat{g}+g^{\top}\left(x^*-x_k\right) \\
			&\geq \delta\|g\|_2+f(x_k)-f^* \geq \delta\|g\|_2,
		\end{aligned}
	\end{equation}
	where the first inequality follows from the fact that $x^* + \delta \hat{g} \in Q$, and the second uses the convexity of $f$.
	
	Next, if $V(x^*, x_k) > \epsilon$, then
	\begin{equation*}
		\begin{aligned}
			& g^{\top} \frac{d_k}{V(s_k, x_k)} \geq g^{\top} \frac{d_k}{D_V} \overset{\textcircled{1}}{\geq} \frac{\delta\|g\|}{D_V} \geq \frac{\delta}{D_V} \cdot g^{\top}\left(\frac{x^*-x_k}{\left\|x^*-x_k\right\|}\right) \cdot \frac{V\left(x^*, x_k\right)}{V\left(x^*, x_k\right)} \\
			&= \frac{\delta}{D_V} \cdot \frac{V\left(x^*, x_k\right)}{\left\|x^*-x_k\right\|} \cdot g^{\top}\left(\frac{x^*-x_k}{V\left(x^*, x_k\right)}\right) \\
			&\geq \frac{\delta}{D} \cdot \frac{V\left(x^*, x_k\right)}{D_V} \cdot g^{\top}\left(\frac{x^*-x_k}{V\left(x^*, x_k\right)}\right) \overset{\textcircled{2}}{\geq} \frac{\delta}{D} \cdot \frac{\varepsilon}{D_V} \cdot g^{\top}\left(\frac{x^*-x_k}{V\left(x^*, x_k\right)}\right),
		\end{aligned}
	\end{equation*}
	where step~\textcircled{1} follows from inequality~\eqref{g_norm_ineq}, and step~\textcircled{2} uses the assumption $V(x^*, x_k) > \epsilon$. In the remaining case, when $V(x^*, x_k) \leq \epsilon$, the desired accuracy has already been achieved.
	
	This concludes the proof.
\end{proof}

Now we prove the linear convergence rate of the Algorithm~\ref{alg:exp_adapt_simplified}.

\begin{theorem}\label{strong_conv_theorem}
	Suppose $f$ is a relatively strongly convex function \eqref{rel_strong_conv}, satisfies relatively smoothness \eqref{def:rs} with the scaling condition property \eqref{scaling_cond_def}. Then, for Algorithm \ref{alg:exp_adapt_simplified}, the following holds:
	\begin{equation*}
		\begin{aligned}
			& f\left(x_k\right)-f^* \leqslant\left(f\left(x_0\right)-f^*\right) \left(\frac{1}{2}\right)^t \cdot \left( 1 - \frac{\gamma_{\min} ^{\gamma_{\min} / (\gamma_{\min} - 1)} }{\gamma_{\min} + 1} \left( \frac{\tau \mu}{2 L_{\max}} \right) ^{1 / (\gamma_{\min}-1)} \right)^{k - t}, \\
		\end{aligned}
	\end{equation*}
	where $\tau > 0$, $t$ is number of times when step size $\alpha_k = 1$, $L_{\max} = \max\limits_{i=0,\dots,k} L_i$ and $\gamma_{\min} = \min\limits_{i=0,\dots,k} \gamma_i$.
\end{theorem}

\begin{proof}
	Consider \eqref{descent_lemma} at iteration $k$:
	\begin{equation*}
		f(x_{k+1}) - f(x_k) \leq \nabla f(x_k)^\top d_k \cdot \min \left\{ \frac{1}{2}, \ \frac{1}{2} \left( \frac{- \nabla f(x_k)^\top d_k}{ 2 L_{\max} V(s_k, x_k) } \right)^{1/(\gamma_{\min} - 1)} \right\},
	\end{equation*}
	Adding and subtracting $f^*$ on the left-hand side, regrouping terms, and denoting $h_k = f(x_k) - f^*$, we obtain
	\begin{align}\label{eq:h_decrease}
		h_{k+1} &\leq h_k \left(1 - \min \left\{ \frac{1}{2}, \frac{1}{2} \left( \frac{- \nabla f(x_k)^\top d_k}{2 L_\text{max} V(s_k, x_k)} \right)^{1/ (\gamma_{\min} - 1)} \right\} \right) \notag \\
		&\leq h_k \left(1 - \min \left\{ \frac{1}{2}, \frac{1}{2} \left( \frac{\tau}{2 L_\text{max}} (- \nabla f(x_k))^\top \left( \frac{x^* - x_k}{V(x^*, x_k)} \right) \right)^{1/ (\gamma_{\min} - 1)} \right\} \right)
	\end{align}
	
	In the last inequality, we used the scaling condition \eqref{scaling_cond_def}.
	
	Next, we apply the relatively strong convexity property \eqref{rel_strong_conv} to derive a lower bound on $f^* - f(x_k)$:
	\begin{equation*}
		\begin{aligned}
			& f^* - f(x_k) \geq \nabla f(x_k)^\top(x^* - x_k) + \mu V(x^*, x_k) \geq \\ 
			& \geq \min_{\alpha_*} \left\{ \alpha_* \nabla f(x_k)^\top (x^* - x_k) + \alpha_*^{\gamma_{\min}} \mu V(x^*, x_k) \right\}
		\end{aligned}
	\end{equation*}
	
	We now determine $\alpha_*$ such that
	\begin{equation} \label{strong_conv_with_alpha}
		\begin{aligned}
			\nabla f(x_k)^\top(x^* - x_k) + \gamma_{\min} \alpha_*^{\gamma_{\min} - 1} \mu V(x^*, x_k) = 0
		\end{aligned}
	\end{equation}
	which yields
	$$\alpha_*^{\gamma_{\min} - 1} = - \frac{-\nabla f(x_k)^\top (x^* - x_k)}{\gamma_{\min} \mu V(x^*, x_k)}$$
	
	Substituting this into \eqref{strong_conv_with_alpha}, we obtain
	
	\begin{equation} \label{precision_strong_conv}
		\begin{aligned}
			& f^* - f(x_k) \geq \left( - \frac{\nabla f(x_k)^\top (x^* - x_k)}{\gamma_{\min} \mu V(x^*, x_k)} \right)^{1 / (\gamma_{\min} - 1)} \nabla f(x_k)^\top(x^* - x_k) + \\ 
			& + \left( - \frac{\nabla f(x_k)^\top (x^* - x_k)}{\gamma_{\min} \mu V(x^*, x_k)} \right)^{\gamma_{\min} / (\gamma_{\min} - 1)} \mu V(x^*, x_k) \geq \\
			& \geq \frac{1 - \gamma_{\min}}{\gamma_{\min}^{\gamma_{\min} / (\gamma_{\min} - 1)}} \cdot \frac{(- \nabla f(x_k)^\top(x^* - x_k))^{\gamma_{\min} / (\gamma_{\min} - 1)}}{(\mu V(x^*, x_k))^{1 / (\gamma_{\min} - 1)}}
		\end{aligned}
	\end{equation}
	
	Substituting the above estimate into \eqref{eq:h_decrease}, we get
	
	\begin{equation} \label{precision_strong_conv}
		\begin{aligned}
			& h_{k+1} \leq h_k \left( 1 - \min \left\{ \frac{1}{2}, \frac{1}{2} \left( \frac{\tau \mu}{2 L_k} \right)^{1/(\gamma_{\min} - 1)} \cdot \frac{\gamma_{\min}^{\gamma_{\min} / (\gamma_{\min} - 1)}}{(1 + \gamma_{\min})} \cdot \frac{- h_k}{\nabla f(x_k)^\top(x^* - x_k)} \right\} \right) \leq \\
			& \leq h_k \left( 1 - \min \left\{ \frac{1}{2}, \frac{1}{2} \left( \frac{\tau \mu}{2 L_k} \right)^{1/(\gamma_{\min} - 1)} \cdot \frac{ \gamma_{\min}^{\gamma_{\min}/(\gamma_{\min} - 1)}}{(\gamma_{\min} + 1)} \right\} \right)
		\end{aligned}
	\end{equation}
	
	where in the last step we used \eqref{fw_gap_ineq}, which allowed us to cancel the terms $\nabla f(x_k)^\top(x^* - x_k)$ and $h_k$.
	
	This concludes the proof.
\end{proof}

Note that if the step size at iteration $k$ is full, i.e., $\alpha_k = 1$, then by Theorem~\ref{strong_conv_theorem} it follows that $f(x_{k+1}) - f^* \leq \frac{f(x_k) - f^*}{2}$; that is, the residual is halved.

Now, combining Lemma~\eqref{scaling_condition_determined} with Theorem~\ref{strong_conv_theorem}, we immediately obtain the following:

\begin{corollary}\label{strong_conv_theorem_determined}
	Suppose $f$ is a relatively strongly convex function \eqref{rel_strong_conv}, satisfies relatively smoothness \eqref{def:rs} with the scaling condition property \eqref{scaling_cond_def}. Then for Algorithm \ref{alg:exp_adapt_simplified}, the following holds:
	\begin{equation*}
		\begin{aligned}
			& f\left(x_k\right)-f^* \leqslant\left(f\left(x_0\right)-f^*\right) \left(\frac{1}{2}\right)^t \left( 1 - \frac{\gamma_{\min} ^{\gamma_{\min} / (\gamma_{\min} - 1)} }{\gamma_{\min} + 1} \left( \frac{\delta}{2 D} \frac{\epsilon}{D_V} \frac{\mu}{2 L_{\max}} \right) ^{1 / (\gamma_{\min}-1)} \right)^{k - t}, \\
		\end{aligned}
	\end{equation*}
	where $\epsilon$ is the desired residual, $\partial Q$ is a boundary of the feasible set, $\delta = \text{dist}(x^*, \partial Q) = \inf_{y \in \partial Q} \| x^* - y \|_2$, $D = \max_{x,y} \| x - y \|_2$, and $D_V = \max_{x,y} V(x, y)$ for $x, y \in Q$.
\end{corollary}

\begin{proof}
	Substitute the value of $\tau$ from Lemma~\ref{scaling_condition_determined} into Theorem~\ref{strong_conv_theorem}.
\end{proof}

\section{Application to Distributed Optimization}

We now consider the distributed optimization problem \eqref{distr_opt_problem} and demonstrate how relative strong convexity \eqref{rel_strong_conv} and relative smoothness \eqref{def:rs} can be applied in the centralized distributed optimization setting. We show that if the objective function satisfies \eqref{rel_strong_conv}, then Algorithm~\ref{alg:exp_adapt_simplified} exhibits acceleration.

The key idea was originally proposed in \cite{hendrikx2020statistically} for gradient methods. We extend this approach to the Frank-Wolfe method.

Assume that all functions \( f_j \) are \( L_\text{euk} \)-smooth and \( \mu_\text{euk} \)-strongly convex in the standard Euclidean sense. Then the aggregate function \( F \) is also \( L_\text{euk} \)-smooth and \( \mu_\text{euk} \)-strongly convex, with \( L_\text{euk}, \mu_\text{euk} > 0 \). 

Additionally, assume the \emph{statistical similarity} property holds, meaning there exists a constant \( \sigma > 0 \) with \( \sigma \ll L_\text{euk} \) such that
\begin{equation}
	\|\nabla F(x) - \nabla \tilde{F}(x) - \nabla F(y) + \nabla \tilde{F}(y)\|^2 \leq \sigma \|x - y\|^2, \quad \forall x, y.
	\label{eq:sigma_assumption}
\end{equation}

Then for any \( x, y \in Q\),
\[
F(y) - \tilde{F}(y) \leq F(x) - \tilde{F}(x) + \langle \nabla F(x) - \nabla \tilde{F}(x), y - x \rangle + \frac{\sigma}{2} \|y - x\|^2,
\]
i.e.,
\[
F(y) \leq F(x) + \langle \nabla F(x), y - x \rangle + \tilde{F}(y) - \tilde{F}(x) - \langle \nabla \tilde{F}(x), y - x \rangle + \frac{\sigma}{2} \|y - x\|^2.
\]
This means that
\begin{equation}
	F(y) \leq F(x) + \langle \nabla F(x), y - x \rangle + V_{\tilde{F}}(y, x) \quad \forall x, y,
	\label{eq:vtilde_def}
\end{equation}
where
\[
V_{\tilde{F}}(y, x) = d_{\tilde{F}}(y) - d_{\tilde{F}}(x) - \langle \nabla d_{\tilde{F}}(x), y - x \rangle, \quad
d_{\tilde{F}}(x) = \tilde{F}(x) + \frac{\sigma}{2} \|x\|^2.
\]

On the other hand, from~\eqref{eq:sigma_assumption} it follows that
\[
F(y) - \tilde{F}(y) \geq F(x) - \tilde{F}(x) + \langle \nabla F(x) - \nabla \tilde{F}(x), y - x \rangle - \frac{\sigma}{2} \|y - x\|^2,
\]
\[
F(y) \geq F(x) + \langle \nabla F(x), y - x \rangle + \tilde{F}(y) - \tilde{F}(x) - \langle \nabla \tilde{F}(x), y - x \rangle - \frac{\sigma}{2} \|y - x\|^2.
\]

Moreover, due to the \( \mu_\text{euk} \)-strong convexity of \( F \), we have
\[
F(y) \geq F(x) + \langle \nabla F(x), y - x \rangle + \frac{\mu_\text{euk}}{2} \|y - x\|^2,
\]
from which for any \( C \in (0, 1) \)
\[
F(y) - F(x) - \langle \nabla F(x), y - x \rangle \geq C \left[ \tilde{F}(y) - \tilde{F}(x) - \langle \nabla \tilde{F}(x), y - x \rangle \right] + \frac{\mu_\text{euk}(1 - C) - C\sigma}{2} \|y - x\|^2.
\]

For \( C = \frac{\mu_\text{euk}}{\mu_\text{euk} + 2\sigma} \), we obtain \( \forall x, y \)
\[
F(y) \geq F(x) + \langle \nabla F(x), y - x \rangle + \frac{\mu_\text{euk}}{\mu_\text{euk} + 2\sigma} V_{\tilde{F}}(y, x).
\]

\textbf{Conclusion:} The assumption \eqref{eq:sigma_assumption} implies that \( F \) is 1-relatively smooth and \( \frac{\mu_\text{euk}}{\mu_\text{euk} + 2\sigma} \)-relatively strongly convex with respect to the Bregman divergence \( V_{\tilde{F}}(y, x) \). Therefore, we can apply Algorithm~\ref{alg:exp_adapt_simplified_distr} with the adaptive step size

\begin{equation}
	\alpha_k := \min \left\{ \left( \frac{- \langle \nabla f(x_k), d_k \rangle }{2 L V_{\tilde{F}}(s_k, x_k)} \right) ^{\frac{1}{\gamma_k - 1}}, 1 \right\}.
\end{equation}

In this setting, the \emph{relative condition number} becomes
\begin{equation}\label{rs_cond_nmbr}
	\frac{1}{\frac{\mu_\text{euk}}{\mu_\text{euk} + 2\sigma}} = 1 + \frac{2\sigma}{\mu_\text{euk}} \quad \left(\ll \frac{L_\text{euk}}{\mu_\text{euk}} \right).
\end{equation}

\begin{algorithm}  
	\caption{Frank-Wolfe Method with Adaptive \( \gamma \)}  
	\label{alg:exp_adapt_simplified_distr}  
	\begin{algorithmic}[1]  
		\Require Initial point \( x_0 \in \operatorname{rint} Q \), \( \gamma_0 = 2 \), \( L > 0 \), \( \delta > 0 \), \( \eta > 1 \)
		\State Initialize: \( x_0 \), \( \gamma_{-1} = \gamma_0 \), \( L_{-1} = L_0 \)  
		\For{\( k = 0, 1, 2, \dots \)}  
		\State \( \gamma_k \gets \min \{ \gamma_k + \eta (\gamma_k - 1), \gamma_{\max} \} \)
		\While{true}
		\State \( s_k \in \text{LMO}_Q (\nabla f(x_k)) \)  
		\State \( d_k = s_k - x_k \)  
		\State \( \alpha_k := \min \left\{ \left( \frac{- \langle \nabla f(x_k), d_k \rangle }{2 L V(s_k, x_k)} \right) ^{\frac{1}{\gamma_k - 1}}, 1 \right\} \)  
		\If{\( V(x_k + \alpha_k d_k) \leq \alpha_k^{\gamma_k} V(s_k, x_k) \)}  \label{alg:exp_adapt_simplified_distr:cond}
		\State \textbf{break}
		\Else  
		\State \( \gamma_k \gets 1 + \frac{\gamma_k - 1}{\eta} \)
		\EndIf  
		\EndWhile
		\State \( x_{k+1} \gets x_k + \alpha_k d_k \)
		\EndFor  
	\end{algorithmic}  
\end{algorithm}

This directly leads to acceleration guaranteed by Theorem~\ref{strong_conv_theorem}, since the convergence rate depends on the relative condition number \eqref{rs_cond_nmbr}, which is significantly smaller than the standard Euclidean counterpart.

Algorithm~\ref{alg:exp_adapt_simplified_distr} adapts only the parameter \( \gamma \), while \( L \) is kept fixed. Note that the adaptation condition~\ref{alg:exp_adapt_simplified_distr:cond} differs from the one used in Algorithm~\ref{alg:exp_adapt_simplified:cond}. This discrepancy arises because the exact function value is not available in the distributed optimization setting.

It is also worth noting that relative smoothness often yields smaller values of the smoothness constant \( L \) compared to the Euclidean case. However, it may come with a larger diameter of the feasible set, which can negatively impact the convergence rate of the Frank-Wolfe algorithm (see the bound in \eqref{theorem:alg_exp_adapt}). As shown in Corollary  \ref{strong_conv_theorem_determined}, this drawback can be mitigated by the favorable denominator \eqref{rs_cond_nmbr}, leading to improved practical performance.

\section{Numerical experiments}

All experiments were conducted using Python 3.11.9 on a computer equipped with an AMD Ryzen 5 5600H 3.30 GHz CPU.  

We performed experiments on relatively smooth applications, specifically the D-optimal experiment design problem (for more details on this problem in the context of relative smoothness, see \cite{lu2018relatively}) and the Poisson linear inverse problem (see \cite{csiszar1991least}). 

We investigate in experiments the next question: does adaptable algorithms outperforms the analogues with fixed parameters? Does adaptable for two parameters algorithm outperforms algorithm adaptable for one parameter?

Also, we measure CPU time consumption, defined as \( t_k - t_0 \), where \( t_0 \) is the starting time of the algorithm and \( t_k \) is the time at iteration \( k \).  

%
%
%
%
%

\subsection{D-Optimal Experiment Design}\label{d_optimal_experiment_design}

D-optimal experiment design aims to select experimental conditions that maximize the determinant of the Fisher information matrix, thereby minimizing the volume of the confidence ellipsoid for the estimated parameters. This approach identifies a subset of experimental conditions that provides the most informative data for parameter estimation in a statistical model.  

It was shown in \cite{lu2018relatively} that the D-optimal experiment design objective is 1-smooth relative to Burg's entropy, defined as  
\[
h(x) = - \sum_{i=1}^{n} \log(x^{(i)})
\]
on \( \mathbb{R}_+^n \).  

We now introduce the objective function and parameters used in our experiments. Given \( n \) vectors \( v_1, \dots, v_n \in \mathbb{R}^m \), where \( n \geq m+1 \), the optimization problem is formulated as follows:  

\begin{equation}\label{objective_exp}
	\begin{aligned} 
		\text{minimize} \quad & f(x) := -\log \det \left( \sum_{i=1}^{n} x^{(i)} v_i v_i^T \right) \\  
		\text{subject to} \quad & \sum_{i=1}^{n} x^{(i)} = 1, \\  
		& x^{(i)} \geq 0, \quad i = 1, \dots, n.
	\end{aligned}
\end{equation}

The initial point is set as \( x_0 = (1/n, \dots, 1/n) \) with \( L=1 \). As mentioned above, Burg's entropy is used as the reference function for the divergence. The vectors \( v_i \) were generated independently from a Gaussian distribution with zero mean and unit variance in \( \mathbb{R}^m \).

The results are presented in Figure~\ref{fig:d_opt_design1} and Figure~\ref{fig:d_opt_design2}. FW-adapt-all (adaptive with respect to both parameters) consistently outperforms FW-adapt-L (adaptive only with respect to $L$), while FW-adapt-L, in turn, consistently outperforms FW-Bregman. Moreover, the more frequently the parameters are adjusted, the more pronounced this difference becomes. This suggests that adaptivity indeed provides acceleration. However, it should be noted that parameter adjustment requires some computational overhead, which is visible in the right-hand side of the plots. It is also worth mentioning that the acceleration effect can be further improved by tuning the adaptation rules, for example, by adjusting the divisor for the parameters $L$ and $\gamma$ at each iteration.



\begin{figure}
	\includegraphics[width=\textwidth]{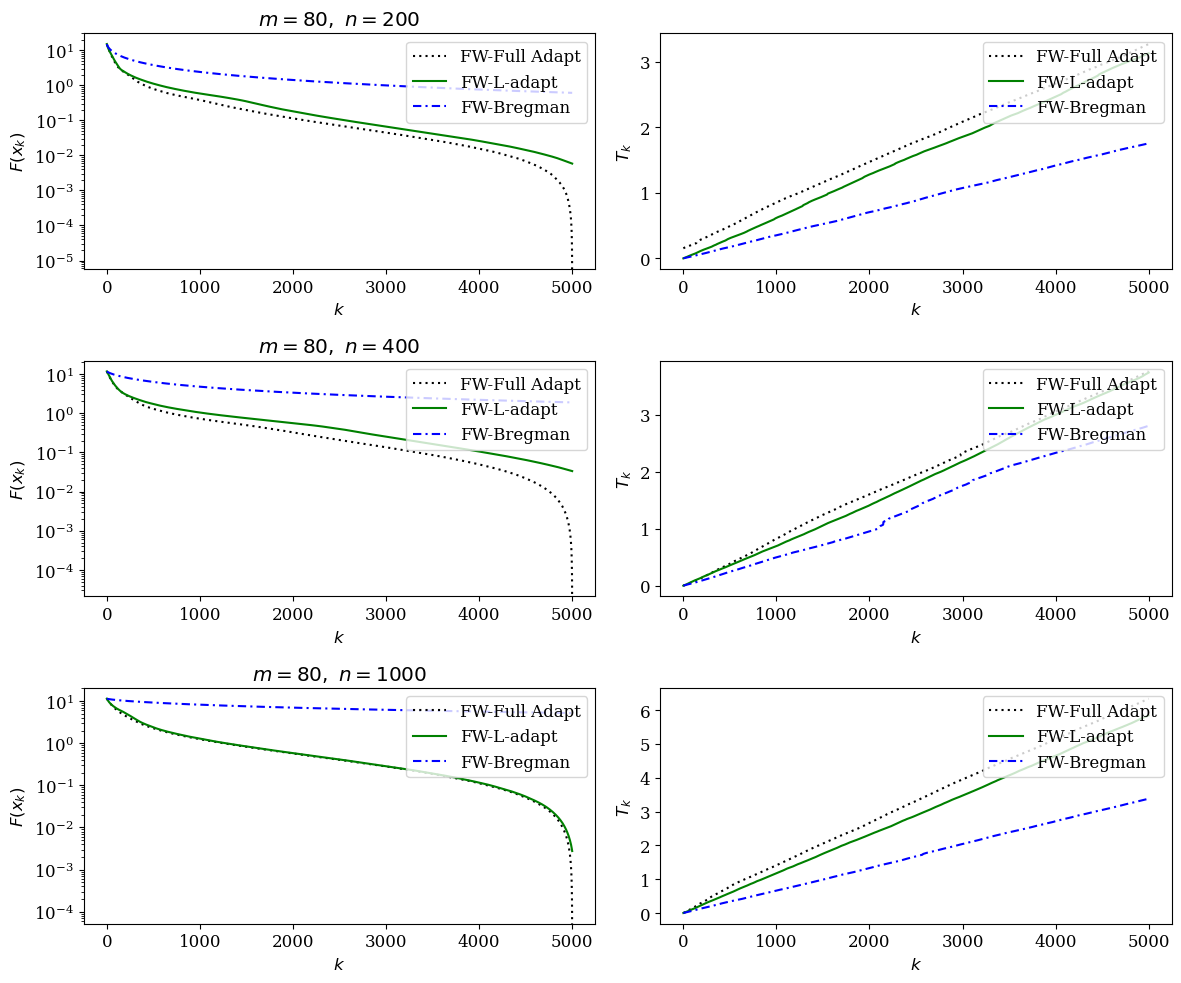}
	\caption{D-optimal experiment design \eqref{objective_exp}, changing the parameter $n$. The left graph shows the convergence rate on a logarithmic scale, while the right graph displays CPU time on a linear scale. FW-Full Adapt corresponds to the step-size \eqref{step_size_l_smooth} adaptable by $L, \gamma$ parameters. FW-adapt refers to step-size \eqref{step_size_l_smooth} adaptable by $L$ parameter and FW-Bregman is step-size \eqref{step_size_l_smooth} with fixed parameters.} \label{fig:d_opt_design1}
\end{figure}

\begin{figure}
	\includegraphics[width=\textwidth]{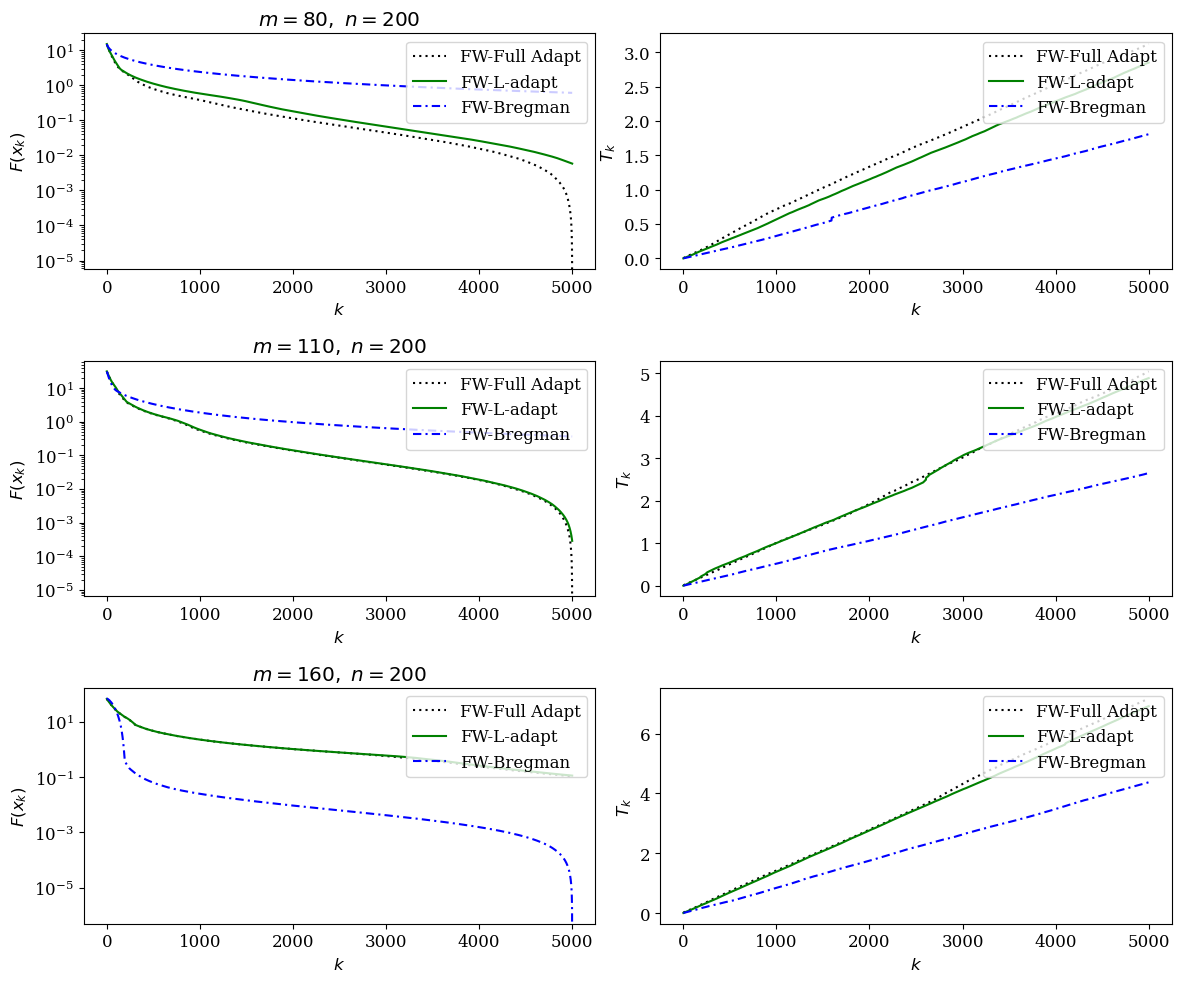}
	\caption{D-optimal experiment design \eqref{objective_exp}, changing the parameter $m$. The left graph shows the convergence rate on a logarithmic scale, while the right graph displays CPU time on a linear scale. FW-Full Adapt corresponds to the step-size \eqref{step_size_l_smooth} adaptable by $L, \gamma$ parameters. FW-adapt refers to step-size \eqref{step_size_l_smooth} adaptable by $L$ parameter and FW-Bregman is step-size \eqref{step_size_l_smooth} with fixed parameters.} \label{fig:d_opt_design2}
\end{figure}

\subsection{Poisson Linear Inverse Problem}
The Poisson linear inverse problem arises in scenarios where observed data follows a Poisson distribution, which is common in counting processes, imaging, and scientific applications such as tomography and photon-limited imaging. The problem is formulated as:

\begin{equation}\label{poisson_linear_inv_problem}
	\begin{aligned}
		\text{minimize} \quad & f(x) :=  \sum_{i=1}^{m} \left(x^{(i)} \log\left(\frac{x^{(i)}}{y^{(i)}} \right) - x^{(i)} + y^{(i)} \right) \\  
		\text{subject to} \quad & \sum_{i=1}^{n} x^{(i)} = 1, \\  
		& x^{(i)} \geq 0, \quad i = 1, \dots, n.
	\end{aligned}
\end{equation}  

Here, \( y \in \mathbb{R}^m_{++} \) represents the observed counts, \( A \in \mathbb{R}^{m \times n}_+ \) is the measurement matrix, and \( x \in \mathbb{R}^n_+ \) is the unknown parameter (e.g., an image or signal). The elements of \( A \) and \( y \) are sampled from a uniform distribution over the interval \( [0, 1] \), with dimensions set to \( m=500 \) and \( n=200 \).  

The objective function \eqref{poisson_linear_inv_problem} is L-relative smooth with respect to \( h(x) = - \sum_{i=1}^{n} \log (x^{{i}}) \) over \( \mathbb{R}^n_{++} \) (see \cite{bauschke2017descent}), with a smoothness constant \( L = \| y \|_1 \). Therefore, we employ the same divergence function as in the previous experiment, defined in \eqref{objective_exp}.

As shown in Figure~\ref{fig:poisson_regr}, we observe almost the same results as in the D-optimal design experiment. The experiment clearly demonstrates that the algorithm adaptive with respect to both $\gamma$ and $L$ consistently achieves faster convergence. The algorithm adaptive with respect to both $\gamma$ and $L$ outperforms the one adaptive only with respect to $L$, which in turn outperforms the non-adaptive version. However, due to the overhead of parameter tuning, these adaptive methods are somewhat slower at the beginning, which is expected since they need time to adjust to the appropriate constants.

\begin{figure}
	\includegraphics[width=\textwidth]{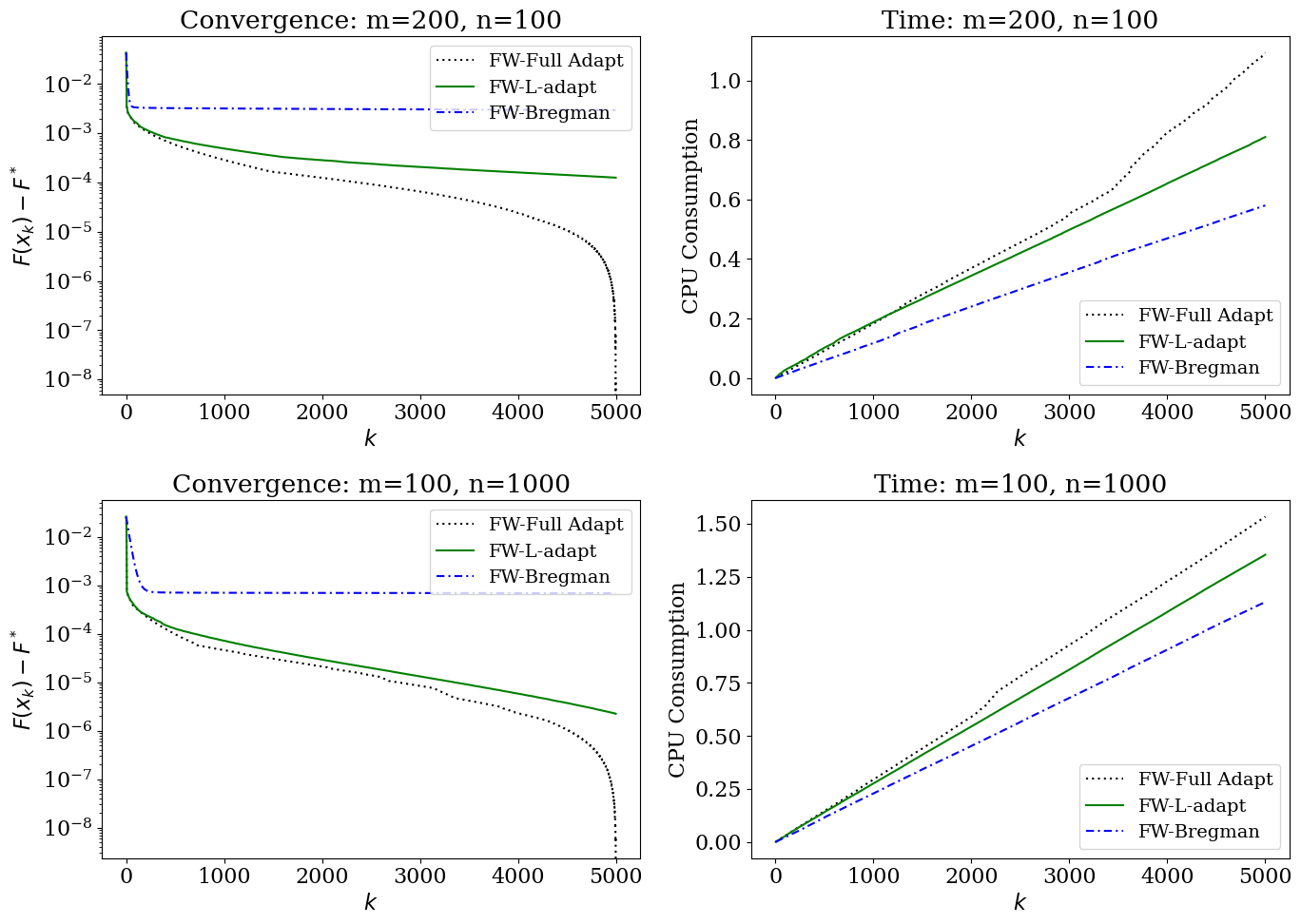}
	\caption{Poisson linear inverse problem \eqref{poisson_linear_inv_problem}. The left graph shows the convergence rate on a logarithmic scale, while the right graph displays CPU time on a linear scale. FW-Full Adapt corresponds to the step-size \eqref{step_size_l_smooth} adaptable by $L, \gamma$ parameters. FW-adapt refers to step-size \eqref{step_size_l_smooth} adaptable by $L$ parameter and FW-Bregman is step-size \eqref{step_size_l_smooth} with fixed parameters.}  
	\label{fig:poisson_regr}
\end{figure}

\section{Conclusion}

In this work, we proposed a fully adaptive variant of the Frank-Wolfe method for optimization problems under relative smoothness. Our approach extends existing adaptive strategies by simultaneously adapting to both the relative smoothness constant \( L \) and the Triangle Scaling Exponent \( \gamma \), without requiring prior knowledge of these parameters.

We also showed how relative smoothness and relative strong convexity naturally emerge in centralized distributed optimization problems. Under a mild variance-type assumption on the gradients, the objective function becomes relatively strongly convex with respect to a local Bregman divergence. This structure allows the use of our adaptive Frank-Wolfe algorithm, which benefits from acceleration due to a significantly improved relative condition number.

\section*{Acknowledgments}
The research was supported by the Russian Science Foundation (project No. 21-71-30005-/pi), \url{https://rscf.ru/en/project/21-71-30005/}.

%
%
%
\bibliographystyle{splncs04}
\bibliography{FW_rs_distributed_and_full_adaptive}

\end{document}